\documentclass[11pt]{amsart}

\newtheorem{case}{Case}
\newtheorem{subcase}{Subcase}[case]

\usepackage{graphicx}
\usepackage{epstopdf}
\usepackage{amsthm,amsmath,amssymb,amsxtra,amscd}
\usepackage{verbatim}
\usepackage{rotating}
\usepackage{multirow}
\usepackage{multicol}
\usepackage{url}
\usepackage{algorithmic}
\usepackage{xypic}
\usepackage{mathtools}

\newtheorem{theorem}{Theorem}[section]

\newtheorem{lemma}[theorem]{Lemma}
\newtheorem{corollary}[theorem]{Corollary}

\newtheorem{definition}[theorem]{Definition}

\newtheorem{proposition}[theorem]{Proposition}

\newtheorem*{example*}{Example}

\newtheoremstyle{myexample}{3pt}{3pt}{\rmfamily}{}{\itshape}{:}{ }{\thmname{#1}\thmnumber{ #2}\thmnote{ (#3)}}
\theoremstyle{myexample}

\newtheoremstyle{myremark}{3pt}{3pt}{\rmfamily}{}{\itshape}{:}{ }{\thmname{#1}}
\theoremstyle{myremark}

\newtheorem*{observation*}{Observation}

\newtheoremstyle{conjecture}{3pt}{3pt}{\itshape}{}{\bfseries}{.}{ }{\thmname{#1}\thmnote{ (#3)}}
\theoremstyle{conjecture}

\newtheorem*{question*}{Question}
\newtheorem{conjecture}{Conjecture}
\newtheorem{theorem*}{Theorem}

\numberwithin{equation}{section}

\newcounter{algorithm}
\renewcommand{\thealgorithm}{\thesection.\arabic{algorithm}}

\begin{document}

\title[About the second neighborhood conjecture]{About the second neighborhood conjecture for tournaments missing two stars or disjoint paths}

\author{Moussa Daamouch}
\address{Department of Mathematics and Physics, Lebanese International University LIU, Saida, Lebanon}
\email{moussadaamouch1923@gmail.com}
\address{KALMA Laboratory, Department of Mathematics, Lebanese University, Beirut, Lebanon}
\email{moussa.daamouch@ul.edu.lb}

\author{Salman Ghazal}
\address{Department of Mathematics and Physics, Lebanese International University LIU, Beirut, Lebanon}
\email{salman.ghazal@liu.edu.lb}
\address{Department of Mathematics, Lebanese University, Beirut, Lebanon}
\email{salman.ghazal@ul.edu.lb}
\address{Department of Mathematics and Physics, The International University of Beirut BIU, Beirut, Lebanon}
\email{salman.ghazal@liu.edu.lb}

\author{Darine Al-Mniny}
\address{KALMA Laboratory, Department of Mathematics, Lebanese University, Beirut, Lebanon}
\email{almniny.darine@gmail.com}
\address{Department of Mathematics and Physics, The International University of Beirut BIU, Beirut, Lebanon}
\email{darine.mniny@liu.edu.lb}
\address{Department of Mathematics and Physics, Lebanese International University LIU, Rayak, Lebanon}
\email{darine.mniny@liu.edu.lb}


\begin{abstract}
Seymour's Second Neighborhood Conjecture (SSNC) asserts that every oriented finite simple graph (without digons) has a vertex whose second out-neighborhood is at least as large as its first out-neighborhood. Such a vertex is said to have the second neighborhood property (SNP). In this paper, we prove SSNC for tournaments missing two stars. We also study SSNC for tournaments missing disjoint paths and, particularly, in the case of missing paths of length 2. In some cases, we exhibit at least two vertices with the SNP.
\end{abstract}

\maketitle

\section{Introduction}
In this paper, a {directed graph} (or a {digraph}) $D$ is a pair of two disjoint sets $(V,E)$, where $E \subset V^2$. $V$ is called the {vertex set} of $D$ and is denoted by $V(D)$. $E$ is called the {edge set} ({arc set}) of $D$ and is denoted by $E(D)$. All the digraphs in this paper are finite oriented graphs (i.e. $V$ is finite, $(u,u) \notin E$ and there is at most one arc between $u$ and $v$ for all $u, v \in V$). The out-neighborhood (resp. in-neighborhood) of a vertex $v$ is denoted by $N_D^+(v)$ or $N^+(v)$ (resp. $N_D^-(v)$ or $N^-(v)$) and the {second out-neighborhood} (resp. second in-neighborhood) of $v$ is denoted by $N_D^{++}(v)$ or $N^{++}(v)$ (resp. $N_D^{--}(v)$ or $N^{--}(v)$). We say that a vertex $v$ has the \emph{second neighborhood property} (SNP), if  $|N^+(v)| \leq |N^{++}(v)|$.  Sometimes, we will use the notation $v^+$, $v^-$ and $v^{++}$ instead of $N^+(v)$, $N^-(v)$ and $N^{++}(v)$ respectively.
In 1990, Paul Seymour proposed the following conjecture.
\begin{conjecture}
In every finite simple digraph, there exists a vertex $v$ such that $|N^+(v)| \leq |N^{++}(v)|$.
\end{conjecture}
It soon became an important topic of interest in graph theory. Although much research was done in that field, SSNC still remains open. It was proven only for some very specific classes of digraphs. In 1995, Dean and Latka \cite{dean1995squaring} conjectured similar statement for tournaments. This problem, known as Dean's conjecture, has been solved in 1996 by Fisher \cite{fisher1996squaring}. In 2000, Havet and Thomass\'e \cite{havet2000median} gave a short proof of Dean's conjecture, using median orders.  A \emph{median order} of a digraph $D$ is a linear order $L=v_1v_2 \ldots v_n$ of its vertex set $V$ such that $|\{(v_i, v_j) : i < j\}|$ (the number of arcs directed from left to right) is as large as possible. The last vertex $v_n$ of a median order $L$ is called a \emph{feed vertex}. Havet and Thomass\'e \cite{havet2000median} proved that, for tournaments, every feed vertex has the SNP. Their proof also yields the existence of two vertices  having the SNP under the condition that no vertex is a \emph{sink} (that is, a vertex of out-degree 0).

\begin{theorem} \cite{havet2000median} \label{t.tourn}
Every tournament with no sink has at least two vertices with the SNP.
\end{theorem}
Unfortunately, for general digraphs it is not guaranteed that a feed vertex has the SNP, see e.g. \cite{havet2000median}.  However, median orders are still used for several cases by applying the completion approach as following. For all $u,v \in V(D)$ such that $(u,v) \notin E(D)$ and $(v,u) \notin E(D)$, we make an arc between $u$ and $v$ in some  proper way to obtain a tournament $T$. Then, we consider a particular median order $L$ of $T$ (clearly, the feed vertex of $L$ has the SNP in $T$) and try to prove that this feed vertex has the SNP in $D$ as well.  In 2007, Fidler and Yuster \cite{fidler2007remarks} used median orders and an other tool called dependency digraph to prove that SSNC holds for tournaments missing a matching. Ghazal \cite{ghazal2012, ghazal2013contribution, ghazal2015remark, ghazal2016about}, also used median orders, dependency digraph and good digraphs, to show that the conjecture holds for some new classes of digraphs (tournaments missing $n$-generalized star and other classes of oriented graphs). In order to generalize the results in \cite{ghazal2012, ghazal2013contribution}, Al-Mniny and Ghazal \cite{al2021second} proved SSNC for tournaments missing a specific graph. Dara et al. \cite{dara2022extending} proved SSNC for tournaments missing a matching and a star, extending results in \cite{fidler2007remarks} and \cite{ghazal2016about}.  The main results in \cite{al2021second, dara2022extending, fidler2007remarks, ghazal2012, ghazal2013contribution, ghazal2015remark,ghazal2016about} are obtained using the completion approach.
Recently, SSNC is proved for some new classes of digraphs (see \cite{cary2019vertices, daamouch2020ant, daamouch2020seymour, daamouch2021seymour, hassan2021seymour}).\\
In this paper, we prove SSNC for tournaments missing two stars. Also, we study SSNC for tournaments missing disjoint paths  of length at most 2 and prove it under some conditions.

\section{Useful Tools}
\subsection{Dependency Digraph $\Delta$}

Let $D$ be an oriented graph. For all $x,y \in V(D)$, if $(x,y) \notin E(D)$ and $(y,x) \notin E(D)$ then $xy$ is called a \emph{missing edge}. A vertex $v$ of $D$ is called a \emph{whole} vertex if $vx$ is not a missing edge of $D$ for all $x\in V(D)$. The \emph{missing graph} $G$ of $D$ is defined to be the graph formed by the missing edges of $D$, formally, $G$ is the graph whose edge set is the set of all the missing edges of $D$ and whose vertex set is the set of the non-whole vertices. In this case, we say that $D$ is missing $G$. Let $xy$ and $ab$ be two missing edges of $D$. We say that $xy$ \emph{loses} to $ab$, and we write $xy \rightarrow ab$ if: $x\rightarrow a$ and $b\notin N^+(x) \cup N^{++}(x)$ as well as $y\rightarrow b$ and $a\notin N^+(y) \cup N^{++}(y)$. The \emph{dependency digraph} of $D$ is denoted by $\Delta(D)$ (or $\Delta$) and is defined as follows: The vertex set of $\Delta$ is $V(\Delta)=\{ab\colon ab\ \ is \ a \ missing\ edge\ of\ D\}$ and the edge set of $\Delta$ is $E(\Delta)=\{(ab,cd)\colon ab \rightarrow cd\}$.

\begin{definition}[Good Missing Edge]\cite{ghazal2012} A missing edge $ab$ of $D$ is called a \emph{good missing edge} if it satisfies (i) or (ii):
\begin{enumerate}
\item[(i)] For all $v\in V(D)\setminus \{a,b\}$, if  $v\rightarrow a$ then $b \in N^+(v) \cup N^{++}(v)$.
\item[(ii)] For all $v\in V(D)\setminus \{a,b\}$, if  $v\rightarrow b$ then $a \in N^+(v) \cup N^{++}(v)$.
\end{enumerate}

If $ab$ satisfies (i), then $(a,b)$ is said to be a \textit{convenient orientation} of $ab$. Else, $(b,a)$ is a convenient orientation of $ab$.
\end{definition}
Note that, by assigning a convenient orientation to a good missing edge $ab$, the out-neighborhood $N^+(v)$ and the second out-neighborhood $N^{++}(v)$ do not modify for any vertex $v \in  V(D)\setminus \{a,b\}$. This fact is useful when we apply the completion approach. 

The following lemma gives a characterization of good missing edges. 
\begin{lemma}\cite{fidler2007remarks}
Let $D$ be an oriented graph and let $\Delta$ be its dependency digraph. A missing edge $ab$ is good if and only if $d_{\Delta}^-(ab)=0$.
\end{lemma}
\subsection{Good Digraph and Good Median Order}
Let $C$ be a connected component of $\Delta$. Set $K(C)=\{u \in V(D)\colon uv \in C$ for some $v \in V(D)\}$. The interval graph of $D$, denoted by $\mathcal{I}_D$ is defined as follows: $V(\mathcal{I}_D)=\{C\colon C$ is a connected component of $\Delta\}$, and $E(\mathcal{I}_D)=\{\{C_1,C_2\}\colon K(C_1)\cap K(C_2)\neq \phi\}$. Let $\xi$ be a connected component of $\mathcal{I}_D$. We set $K(\xi)=\cup_{C \in \xi}K(C)$. Note that if $uv$ is a missing edge of $D$, then there is a unique connected component $\xi$ of $\mathcal{I}_D$ such that $u,v \in K(\xi).$
Let $f\in V(D)$, we set
\begin{displaymath}
J(f)=\begin{cases}
    \{f\} & \text{if $f$ is a whole vertex;}\\
    K(\xi) & \text{otherwise, where $f \in K(\xi)$}.
\end{cases}    
\end{displaymath}
Clearly, if $x \in J(f)$, then $J(f) = J(x)$. While if $x \notin J(f)$, then $x$ is adjacent to every vertex in $J(f)$.

Note that, in this paper, we only consider non weighted digraphs. However, we need some prerequisites that are obtained on vertex weighted digraphs. For this reason, we introduce them here. It is clear that the results obtained on weighted digraphs can be used for non weighted digraphs by taking the weight of every vertex equals 1. Let $D=(V,E)$ be an oriented graph and let $\omega: V \rightarrow \mathcal{R}_+$ be a strictly positive real valued function. The couple $(D,\omega)$ is called a {weighted digraph}. Let $K\subseteq V(D)$, $K$ is called an \emph{interval} of $D$ if for all $ u,v \in K$ we have $N^+(u)\setminus K = N^+(v)\setminus K$ and $N^-(u)\setminus K = N^-(v)\setminus K$. We say that $(D,\omega)$ is a \emph{good digraph} if the sets $K(\xi)$'s are intervals of $D$.\\
For $S \subseteq V$, we define the weight of $S$ as $\omega(S)=\displaystyle\sum_{x\in S}\omega(x)$. We define the weight of an arc $e=(u,v)$ by $\omega(e)=\omega(u) \times \omega(v)$. A \emph{weighted median order} of a digraph $(D,\omega)$ is a linear order $L=v_1v_2 \ldots v_n$ of its vertex set $V$ such that $\omega(\{(v_i, v_j) : i < j\})$  is as large as possible.\\
The following sets are called interval of $L$: $[v_i,v_j]:=[i,j]:=\{v_i,\ldots,v_j\}$ and $]v_i,v_j[ := ]i,j[ := \{v_{i+1},\ldots,v_{j-1}\}$. We may sometimes write $[i,j]$ instead of $D[i,j]$.

\begin{lemma}\cite{ghazal2015remark}
Let $(D,\omega)$ be a good digraph. There exists a weighted median order $L=x_1,\ldots,x_n$ such that the $K(\xi)$'s form intervals of $L$. Such a weighted median order $L$ is called good weighted median order of $D$.
\end{lemma}

Let $L = v_1v_2\ldots v_n$ be a weighted median order. Among the vertices not in $N^+(v_n)$, two types are distinguished: A vertex $v_j$ is \emph{good} if there is $i< j$ such that $v_n \rightarrow v_i \rightarrow v_j$, otherwise $v_j$ is a \emph{bad vertex}. The set of good vertices of $L$ is denoted by $G_L^D$ (or $G_L$ if $D$ is clear in the context) \cite{fidler2007remarks}. Clearly, $G_L \subseteq N^{++}(v_n).$ The notion of good vertices is essential for the next theorem, and the notion of bad vertices is used in this paper to create a specific median order from another one as well as we use it in the proof of Theorem \ref{t.nosink}.
\begin{theorem} \cite{ghazal2015remark}\label{th1}
Let $(D,\omega)$ be a good weighted digraph and let $L=x_1,\ldots,x_n$ be a good weighted median order of $(D,\omega)$. For all $x \in J(x_n)$, we have $\omega(N^+(x)\setminus J(x_n)) \leq \omega(G_L \setminus J(x_n)).$
\end{theorem}
We say that a vertex $v$ has the weighted SNP if $\omega(N^+(v)) \leq \omega(N^{++}(v))$. By the previous theorem,  if $x$ has the weighted SNP in $(D[J(f)], \omega)$, then it has the weighted SNP in $(D, \omega)$. Furthermore,  the completion approach can now be refined as follows. We orient some missing edges of $D$ to obtain a good digraph $D'$ (not necessarily a tournament). Then we consider a good median order of feed vertex $f$, and find a vertex $x$ having the SNP in $D'[J(f)]$. Finally, we try to prove that $x$ has the SNP in $D$ as well. 

\begin{definition}[Good Completion]
Let $D$ and $D'$ be two digraphs. We say that $D'$ is a good completion of $D$ if $V(D')=V(D)$, $E(D) \subseteq E(D')$ and $D'$ is a good digraph.
\end{definition}
\begin{theorem} \cite{ghazal2015remark}
\label{t.gazal}
Let $D$ be an oriented graph missing a matching. There is a good completion $D'$ of $D$ such that, for all $ \ f$ feed vertex of $D'$, $f$ has the SNP in $D'$ and in $D$.
\end{theorem}
\noindent {\bf Remark:} If $D$ is an oriented graph missing a matching, then the dependency digraph of $D$ is composed of vertex disjoint directed paths and directed cycles \cite{fidler2007remarks}.

Using same procedures for the proof of the previous theorem, we prove the following.
\begin{theorem} \label{th_good_completion}
Let $D$ be a digraph, and let $\Delta$ denote its dependency digraph. Suppose that for all $x \in V(D)$, we have (i) or (ii) where:\\
\textbf{(i)} $J(x)= K(P)$ for some directed path $P$ in $\Delta$.\\
\textbf{(ii)} $J(x)$ is an interval of $D$ such that there exists $  p \in J(x)$ and $p$ satisfies the SNP in $D[J(x)]$.\\
Then there is a good completion $D'$ of $D$ such that for all $f$ feed vertex of $D'$, there exists $ p \in J_{D'}(f)$ such that $p$ has the SNP in $D'$ and in $D$.
\end{theorem}
We will give the proof of Theorem \ref{th_good_completion} in section \ref{proof}.

\section{SSNC for tournaments missing two stars}
\label{main1}
For a non-negative integer $k$, a graph whose vertex-set $\{x, a_1, a_2,\ldots, a_k\}$ and whose edge-set $\{a_ix \colon i=1,\ldots,k\}$ is called a star of center $x$ and leaves $\{a_1, a_2,\ldots, a_k\}$ and is denoted by $S_x$. Two stars $S_x$ and $S_y$ with $x \neq y$ are said to be disjoint if they do not share a common vertex. Otherwise, they are said to be non-disjoint.\\
First, we introduce a particular order, obtained from a median order $L$ following a specific rearrangement. This new order, denoted by $Sed(L)$, is called the \emph{sedimentation} of $L$. In this section, we will use $Sed(L)$ to prove SSNC for tournaments missing two stars. Also, $Sed(L)$ is useful to exhibit at least two vertices with the SNP in subsection \ref{ssnc_leq2}.\\ 
Let $L$ be a good weighted median order of a good digraph $(D,\omega)$ and let $f$ denote its feed vertex. By Theorem \ref{th1}, for every $x \in J(f)$, we have $\omega(N^+(x)\setminus J(f)) \leq \omega(G_L \setminus J(f))$. Let $b_1,\ldots,b_r$ denote the bad vertices of $L$ not in $J(f)$ and $v_1,\ldots,v_s$ denote the non bad vertices of $L$ not in $J(f)$, both enumerated in increasing order with respect to their index in $L.$ If $\omega(N^+(f)\setminus J(f)) < \omega(G_L \setminus J(f))$, we set $Sed(L) = L.$ If $\omega(N^+(f) \setminus J(f)) = \omega(G_L \setminus J(f))$, we set $sed(L) = b_1\ldots b_rJ(f)v_1\ldots v_s$. 

\begin{lemma} \cite{ghazal2015remark}
Let $L$ be a good weighted median order of a good weighted digraph $(D,\omega)$. We have $Sed(L)$ is a good weighted median order of $(D,\omega).$
\end{lemma}

\begin{theorem}
\label{disjointstars}
If $D$ is an oriented graph missing two stars $S_x$ and $S_y$, then $D$ satisfies SSNC.
\end{theorem}
\begin{proof}
We will consider first the case when $S_x$ and $S_y$ are disjoint. Without loss of generality, we assume that $(y,x) \in E(D)$. Let $D'$ be the digraph obtained from $D$ by removing $y$. Clearly, the missing graph of $D'$ is $S_x$, and so all the missing edges are good.  Assign to each missing edge of $D'$ a convenient orientation. The obtained oriented graph is a tournament $T$. Let $L$ be a median order of $T$ that maximizes $\alpha$, the index of $x$ in $L$,  and let $f$ denote its feed vertex. Note that $T$ is a good digraph since $J(x)=\{x\}$ for every $x \in V(T)$. Furthermore, every median order $L$ of $T$ is a good median order and $Sed(L)$ is also a median order of $T$. By Theorem \ref{th1}, we have $|N_{T}^{+}(f)| \leq |G_L^T|$. In what follows, we will prove that $f$ satisfies the SNP in $D$. To this end, we consider the possible positions of the arc $(f,y)$.
\begin{case}
$(f,y) \notin E(D)$. Here we have two cases:
\begin{subcase}
{$f \notin S_x$.}\\
It is easy to see that $N^{+}_D(f)=N^{+}_T(f)$ and $G_L^T \subseteq N^{++}_D(f)$. Combining these two facts with the fact that  $|N_{T}^{+}(f)| \leq |G_L^T|$, we get that $f$ has the SNP in $D$.
\end{subcase}
\begin{subcase}
{$f \in S_x$.}\\
Reorient all the missing edges incident to $f$ towards $f$ (if any). Hence $L$ is a median order of the new tournament $T'$ and $|N_{T'}^{+}(f)| \leq |G_L^{T'}|$. Note that  $N^{+}_D(f)=N^{+}_{T'}(f)$ and $G_L^{T'} \subseteq N^{++}_D(f)$. All these together imply that $f$ has the SNP in $D$.
\end{subcase}
\end{case}
\begin{case}
{$(f,y) \in E(D)$.}\\
Note that $f \neq x$ and $f \notin S_y$. We proceed as above by considering the possible positions of the vertex $f$:
\begin{subcase}\label{s1}
{$f \notin S_x$.}\\
Observe that $N^{+}_D(f)=N^{+}_T(f) \cup \{y\}$, $G_L^T \subseteq N^{++}_D(f)$ and $f\rightarrow y \rightarrow x$ in $D$. If $x \in N^{+}_{T}(f) \cup G_{L}^{T}$ and   $|N_{T}^{+}(f)| =|G_L^T|$, then $sed(L)$ is a median order of $T$ in which the index of $x$ is greater than $\alpha$, a contradiction. This implies that  either $x \notin N^{+}_{T}(f) \cup G_{L}^{T}$ or  $|N_{T}^{+}(f)| <|G_L^T|$. In the former case, we have  $G_L^{T} \cup \{x\} \subseteq N^{++}_D(f)$ as $f\rightarrow y \rightarrow x$ in $D$, and so $|N_{D}^{+}(f)|= |N_{T}^{+}(f)| +1 \leq |G_L^T|+1 \leq |N_{D}^{++}(f)|$. In the latter case, we  have $|N_{D}^{+}(f)|= |N_{T}^{+}(f)| +1 \leq |G_L^T| \leq |N_{D}^{++}(f)|$.
\end{subcase}
\begin{subcase}\label{s2}
{$f \in S_x$.}\\
If $(x,f) \in E(T)$, then we proceed as in Subcase \ref{s1}. If $(f,x) \in E(T)$, we reorient the edge $xf$ towards $f$, then $L$ is a median order of the new tournament $T'$ and $|N_{T'}^{+}(f)| \leq |G_L^{T'}|$. Note that  $N^{+}_D(f)=N^{+}_{T'}(f) \cup \{y\}$, $G_L^{T'} \subseteq N^{++}_D(f)$ and $x  \notin N^{+}_{T'}(f)$. If $x  \in G_{L}^{T'}$ and $|N_{T'}^{+}(f)| =|G_L^{T'}|$, then $sed(L)$ is a median order of $T'$ in which the index of $x$ is greater than $\alpha$ and also greater than the index of $f$. The latter gives that $(x,f)$ is a backward arc (directed from right to left) in $T'$ with respect to $sed(L)$. Reassigning to the edge $xf$ its initial orientation, we get back to the tournament $T$ such that $sed(L)$ is a median order of $T$, a contradiction to the fact that $L$ maximizes the index of $x$. This implies that either $x \notin G_{L}^{T'}$ and hence  $G_L^{T'} \cup \{x\} \subseteq N^{++}_D(f)$ as $f\rightarrow y \rightarrow x$ in $D$, or  $|N_{T'}^{+}(f)| <|G_L^{T'}|$. In both cases, we get that $|N_{D}^{+}(f)| \leq |N_{D}^{++}(f)|$.   
\end{subcase}
\end{case}
This completes the proof of the case when $S_x$ and $S_y$ are disjoint.\vspace{3mm}\\
	\noindent Now we will study the case when $S_x$ and $S_y$ are non-disjoint.\vspace{3mm}\\
	\noindent We will suppose first that only the two centers are adjacent, that is $xy$ is a missing edge.  Assume without loss of generality that $(y,x)$ is a convenient orientation of the good missing edge $xy$ of $D$. The proof can be done by imitating the case when $S_x$ and $S_y$ are disjoint, with exactly two  differences. The  first difference is that $yx$ is a missing edge of $D$ whose convenient orientation is $(y,x)$. The second difference is that in Subcase \ref{s1} and Subcase \ref{s2}, in case that $x \notin N^{+}_{T}(f) \cup G_{L}^{T}$, we get that $G_L^{T} \cup \{x\} \subseteq N^{++}_D(f)$ because $xy$ is a good missing edge of $D$ and $(f,y) \in E(D)$.\vspace{3mm}\\
	\noindent To end the proof, it remains to confirm SSNC for the case when the set of the common vertices is a subset of the leaves of $S_x$ and $S_y$ or the centers $x$ and $y$. Indeed, this case can be proved by following the overall proof of the above two cases.
\end{proof}
\section{SSNC for tournaments missing disjoint paths}
\label{main2}
Let $D$ be an oriented graph missing disjoint paths and let $\Delta$ denote its dependency digraph.

\subsection{Double Cycles in $\Delta$}\label{double.cycle}

\begin{lemma}\label{lm1}
Let $D$ be a digraph missing disjoint paths. Let $ab$, $xy$, and $zt$ be three missing edges of $D$. If $ab \rightarrow xy$ and $ab \rightarrow zt$, then $\{{x,y}\}\cap\{{z,t}\}\neq\emptyset$. That is, $xy$ and $zt$ are adjacent.
\end{lemma}
\begin{proof}
Since $ab \rightarrow xy$, we have ${a\rightarrow x}$ and ${b\rightarrow y}$ where $x \notin b^+\cup b^{++}$
and $y \notin a^+\cup a^{++}$. Also ${a\rightarrow z}$ and ${b\rightarrow t}$ where $z \notin b^+\cup b^{++}$ and $t \notin a^+\cup a^{++}$ since $ab \rightarrow zt$. Suppose that $\{x,y\}\cap\{z,t\}=\emptyset$. It follows that $xt$ or $yz$ is not a missing edge. We may suppose that $yz$ is not a missing edge. If $y\rightarrow z$, then $ b\rightarrow y \rightarrow z$ and hence $z\in b^{++}$, a contradiction. If $z\rightarrow y$, then $ a\rightarrow z \rightarrow y$ and hence $y\in a^{++}$, a contradiction. Thus $\{{x,y}\}\cap\{{z,t}\}\neq\emptyset$.
\end{proof}

\begin{lemma}\label{lm2}
Let $D$ be a digraph missing disjoint paths. Let $ab$, $xy$, and $zt$ be three missing edges of $D$. If $xy \rightarrow ab$ and $zt \rightarrow ab$, then $\{{x,y}\}\cap\{{z,t}\}\neq\emptyset.$ That is, $xy$ and $zt$ are adjacent.
\end{lemma}
\begin{proof}
Since $xy \rightarrow ab$, we have ${x\rightarrow a}$ and ${y\rightarrow b}$ where $a \notin y^+\cup y^{++}$
and $b \notin x^+\cup x^{++}$. Similarly, since $zt \rightarrow ab$, we have ${z\rightarrow a}$ and ${t\rightarrow b}$ where $a \notin t^+\cup t^{++}$ and $b \notin z^+\cup z^{++}$. Suppose that $\{x,y\}\cap\{z,t\}=\emptyset$. It follows that $xt$ or $yz$ is not a missing edge. We may suppose that $yz$ is not a missing edge. If $y\rightarrow z$, then $ y\rightarrow z \rightarrow a$ and hence $a\in y^{++}$, a contradiction. If $z\rightarrow y$, then $ z\rightarrow y \rightarrow b$ and hence $b\in z^{++}$, a contradiction. Thus, $\{{x,y}\}\cap\{{z,t}\}\neq\emptyset$
\end{proof}

\begin{lemma}\label{lm3}
Let $D$ be a digraph missing disjoint paths and let $\Delta$ denote its dependency digraph. Let $abc$ and $xyz$ be two disjoint missing paths of $D$ of length 2 such that $d^+_{\Delta}(ab)=d^+_{\Delta}(bc)=2$ and $d^-_{\Delta}(xy)=d^-_{\Delta}(yz)=2$. If $ab\rightarrow xy$, then $ab\rightarrow yz$ and $bc\rightarrow xy$ as well as $bc\rightarrow yz$.
\end{lemma}
\begin{proof}
We have $ab\rightarrow xy$ and $d^+(ab)=2$, so $ab\rightarrow yz$ by Lemma \ref{lm1}. But $d^-(xy)=2$ and $ab\rightarrow xy$. Hence $bc\rightarrow xy$ by Lemma \ref{lm2}. By the same justification, we get $bc\rightarrow yz$.
\end{proof}

We may write $ab \rightarrow xyz$ when $ab \rightarrow xy$ and $ab \rightarrow yz$. Also we write $abc \rightarrow xyz$  when $ab \rightarrow xyz$ and $bc \rightarrow xyz$.

\begin{proposition}
\label{p.in-out_leq_2}
If $D$ is a digraph missing disjoint paths, then the maximum out-degree and the maximum in-degree in $\Delta(D)$ are at most 2.
\end{proposition}
\begin{proof}
Let $e$ be a missing edge of $D$. Suppose that $e \rightarrow e_1$ and $e \rightarrow e_2$ where $e_1$ and $e_2$ are two missing edges of $D$. Hence, $e_1$ and $e_2$ are adjacent by Lemma \ref{lm1}. If $e \rightarrow e_3$, then $e_3$ must be adjacent to $e_1$ and $e_2$, a contradiction. Thus, the maximum out-degree is at most 2.\\
By the same justification and by using Lemma \ref{lm2}, we get that the maximum in-degree in $\Delta$ is at most 2.
\end{proof}

\begin{definition}
Let $\{a_ib_ic_i\colon i=1,\ldots,k\}$ be a set of disjoint missing paths of a digraph $D$. We say that $C=a_1b_1c_1, \ldots, a_kb_kc_k$ is a \emph{double cycle} in $\Delta(D)$ if $a_1b_1c_1 \rightarrow a_2b_2c_2 \rightarrow \cdots  \rightarrow a_kb_kc_k \rightarrow a_1b_1c_1$.
\end{definition}

\subsection{Tournament missing disjoint paths of length 2}\label{t.p}
Let $D$ be a tournament missing disjoint paths of length 2 and let $\Delta$ denote its dependency digraph. Suppose that $\Delta$ is 2-regular. That is, for every missing edge $ab$ in  $\Delta$, we have $d^+_{\Delta}(ab)=d^-_{\Delta}(ab)=2$. By the previous section, we get that $\Delta$ is composed only of {double cycles}. In this particular case, we prove that {SSNC} holds. Let $C=a_1b_1c_1, \ldots, a_kb_kc_k$ be a double cycle in $\Delta$. Set $[1,k]=\{1,\ldots,k\}$.

\subsubsection{SSNC in $D[K(C)]$}
Throughout this subsection, the subscripts are taken modulo $k$, and a subscript 0 is considered to be $k$.  For $x \in K(C)$, we may write $x^+, x^-,$ and $x^{++}$ instead of $N_{D[K(C)]}^+(x), N_{D[K(C)]}^-(x),$ and $N_{D[K(C)]}^{++}(x)$ respectively.
\begin{lemma}
\label{lm4}
Let $D$ be a tournament missing disjoint paths of length 2. Let $C=a_1b_1c_1, \ldots, a_kb_kc_k$ be a double cycle in $\Delta(D)$. Set $\{x_i,y_i\} = \{a_i,b_i\}$ or $\{x_i,y_i\} = \{b_i,c_i\}$ for all $i\in [1,k]$. Let $j\in [1,k]$. For all $i\in [1,k]-\{j\}$, we have:
\begin{enumerate}
\item If $x_j \rightarrow x_i$, then $y_i \rightarrow x_j$.
\item If $x_j \rightarrow y_i$, then $x_i \rightarrow x_j$.
\item If $y_j \rightarrow x_i$, then $y_i \rightarrow y_j$.
\item If $y_j \rightarrow y_i$, then $x_i \rightarrow y_j$.
\end{enumerate}
\end{lemma}
\begin{proof}
It is sufficient to prove it for $j=1$ as there is no loss of generality.\\
(1) The proof is by induction on $i$. For $i=2$, the statement is true by the definition of losing relations. Suppose that it is true for $i\geq 2$ and let us prove it for $i+1$ (for $i=k$, we take $i+1$ modulo $k$). So let  $x_1 \rightarrow x_{i+1}$ and we will show that $y_{i+1} \rightarrow x_1$.\\
We have $x_iy_i\rightarrow x_{i+1}y_{i+1}$. We may suppose that $ x_i\rightarrow x_{i+1}$ and $y_i\rightarrow y_{i+1}$. We can easily show that $x_i \rightarrow x_1$. On the contrary, suppose that $x_i \nrightarrow x_1$; this means that $x_1 \rightarrow x_i$  since $x_i$ and $x_1$ are adjacent. Hence $y_i \rightarrow x_1$ by the induction hypothesis. So $y_i\rightarrow x_1 \rightarrow x_{i+1}$, and hence $x_{i+1} \in y_i^{++}$, which contradicts the losing relation $x_iy_i\rightarrow x_{i+1}y_{i+1}$. Thus, $x_i \rightarrow x_1$. If $x_1 \rightarrow y_{i+1}$, then $x_i \rightarrow x_1 \rightarrow y_{i+1}$. Hence $y_{i+1} \in x_i^{++}$, which contradicts the losing relation $x_iy_i\rightarrow x_{i+1}y_{i+1}$. Therefore, $y_{i+1} \rightarrow x_1$.\\
(2) The proof is done by switching $x_i$ and $y_i$.\\
(3) and (4) We replace $x_1$ by $y_1$.
\end{proof}

\begin{corollary}
\label{c.relation}
Let $D$ be a tournament missing disjoint paths of length 2. Let $C=a_1b_1c_1, \ldots, a_kb_kc_k$ be a double cycle in $\Delta(D)$. Set $\{x_i,y_i\} = \{a_i,b_i\}$ or $\{x_i,y_i\} = \{b_i,c_i\}$ for all $i\in [1,k]$. Let $j\in [1,k]$. For all $i\in [1,k]-\{j\}$, we have:
\begin{enumerate}
\item $x_j \rightarrow x_i$ if and only if $y_i \rightarrow x_j$.
\item $x_j \rightarrow y_i$ if and only if $x_i \rightarrow x_j$.
\item $y_j \rightarrow x_i$ if and only if $y_i \rightarrow y_j$.
\item $y_j \rightarrow y_i$ if and only if $x_i \rightarrow y_j$.
\end{enumerate}
\end{corollary}
\begin{proof}
(1) By Lemma \ref{lm4}, if $x_j \rightarrow x_i$, then $y_i \rightarrow x_j$. Conversely,  if $y_i \rightarrow x_j$, then $y_j \rightarrow y_i$, hence $x_i \rightarrow y_j$, and finally $x_j \rightarrow x_i$.\\
(2), (3) and (4) We apply the same reasoning.
\end{proof}

\indent In view of what precedes, we obtain the following conclusion:\\
\textbf{Conclusion}: Let $x \in K(C)$ and $uv$ be a missing edge in $C$ such that $xu$ and $xv$ are not missing edges. We have $x \rightarrow u$ if and only if $x\leftarrow v$.

\begin{lemma}
\label{lm5}
Let $D$ be a tournament missing disjoint paths of length 2. Let $C=a_1b_1c_1, \ldots, a_kb_kc_k$ be a double cycle in $\Delta(D)$. For all $i\in [1,k]$, we have:
\begin{enumerate}
  \item $N^+_{D[K(C)]}(a_i) \setminus \{{c_i}\}=N^+_{D[K(C)]}(c_i) \setminus \{{a_i}\}$.
  \item $N^-_{D[K(C)]}(a_i) \setminus \{{c_i}\}=N^-_{D[K(C)]}(c_i) \setminus \{{a_i}\}$.
\end{enumerate}
\end{lemma}
\begin{proof}
(1) Let $j \in [1,k] \setminus \{i\}$. By Corollary \ref{c.relation}, we have $a_i \rightarrow a_j$ if and only if $a_j \rightarrow b_i$, that is, if and only if $c_i \rightarrow a_j$. Similarly, we have $a_i \rightarrow b_j$ if and only if $c_i \rightarrow b_j$, and also $a_i \rightarrow c_j$ if and only if $c_i \rightarrow c_j$. It follows that $N^+_{D[K(C)]}(a_i) \setminus \{{c_i}\}=N^+_{D[K(C)]}(c_i) \setminus \{{a_i}\}$.\\
(2) Likewise, by Corollary \ref{c.relation}, we get $N^-_{D[K(C)]}(a_i) \setminus \{{c_i}\}=N^-_{D[K(C)]}(c_i) \setminus \{{a_i}\}$.
\end{proof}

\begin{lemma}
\label{lm++}
Let $D$ be a tournament missing disjoint paths of length 2. Let $C=a_1b_1c_1, \ldots, a_kb_kc_k$ be a double cycle in $\Delta(D)$. For all $i\in [1,k]$, we have:
\begin{enumerate}
  \item $N^{++}_{D[K(C)]}(a_i)=N^{++}_{D[K(C)]}(c_i)$.
  \item $N^{--}_{D[K(C)]}(a_i)=N^{--}_{D[K(C)]}(c_i)$.
\end{enumerate}
\end{lemma}
\begin{proof}
(1) Let $y \in N^{++}_{D[K(C)]}(a_i)$. So, there exists $x \in N^+_{D[K(C)]}(a_i)$ such that $a_i \rightarrow x \rightarrow y$ and $y \rightarrow a_i$. Note that $x \neq c_i$ since otherwise $c_i \rightarrow y$ implying that $a_i \rightarrow y$ by Lemma \ref{lm5}, a contradiction. Likewise, we can see that $y \neq c_i$. Hence, by Lemma \ref{lm5}, we get $c_i \rightarrow x \rightarrow y$ and $y \rightarrow c_i$. Thus, $N^{++}_{D[K(C)]}(a_i) \subseteq N^{++}_{D[K(C)]}(c_i)$. The converse is proved similarly.\\
(2) By applying the same reasoning, we get $N^{--}_{D[K(C)]}(a_i)=N^{--}_{D[K(C)]}(c_i)$.
\end{proof}

By Corollary \ref{c.relation} and Lemma \ref{lm5}, for $i \in [1,k]$, we may see that $a_i$ and $c_i$ have the same behavior which is the converse of that of $b_i$. This fact is restated more precisely in the following corollary.
\begin{corollary}
\label{c.p}
Let $D$ be a tournament missing disjoint paths of length 2. Let $C=a_1b_1c_1, \ldots, a_kb_kc_k$ be a double cycle in $\Delta(D)$. For all $t \in [1,k]$ and $i \in [1,k]-\{t\}$, in $D[K(C)]$, we have:
\begin{enumerate}\item The following statements are equivalent.
	  \begin{itemize}
	  \item $a_i \in \{a_t^+ \cup c_t^+ \cup b_t^-\}$
	  \item $c_i \in \{a_t^+ \cup c_t^+ \cup b_t^-\}$
	  \item $\{a_i,c_i\} \subseteq \{a_t^+ \cap c_t^+ \cap b_t^-\}$
	  \item $b_i \in \{a_t^- \cup c_t^- \cup b_t^+\}$
	  \item $b_i \in \{a_t^- \cap c_t^- \cap b_t^+\}$
	  \end{itemize}	 
\item The following statements are equivalent.
	  \begin{itemize}
	  \item $a_i \in \{a_t^- \cup c_t^- \cup b_t^+\}$
	  \item $c_i \in \{a_t^- \cup c_t^- \cup b_t^+\}$
	  \item $\{a_i,c_i\} \subseteq \{a_t^- \cap c_t^- \cap b_t^+\}$
	  \item $b_i \in \{a_t^+ \cup c_t^+ \cup b_t^-\}$
	  \item $b_i \in \{a_t^+ \cap c_t^+ \cap b_t^-\}$
	  \end{itemize}
\end{enumerate}
\end{corollary}

\begin{lemma}
\label{lm6}
Let $D$ be a tournament missing disjoint paths of length 2. Let $C=a_1b_1c_1, \ldots, a_kb_kc_k$ be a double cycle in $\Delta(D)$. For all $t \in [1,k]$, we have:
\begin{enumerate}
\item $a_t^{++} = a_t^- \cup \{b_t\} \setminus \{a_{t+1},b_{t+1},c_{t+1},c_t\}$.
\item $c_t^{++} = c_t^- \cup \{b_t\} \setminus \{a_{t+1},b_{t+1},c_{t+1},a_t\}$.
\item $b_t^{++} = b_t^- \cup \{a_t,c_t\} \setminus \{a_{t+1},b_{t+1},c_{t+1}\}$.
\end{enumerate}
\end{lemma}
\begin{proof}
(1) From the losing relations $a_tb_t \rightarrow a_{t+1}b_{t+1}$ and $a_tb_t \rightarrow b_{t+1}c_{t+1}$, we have $\{a_{t+1},b_{t+1},c_{t+1}\} \cap a_t^{++} =\emptyset$ and $b_t \in a_t^{++}$. Also, we have $c_t \notin a_t^{++}$ because otherwise there exists $x \in K(C)$ such that $a_t \rightarrow x \rightarrow c_t$, which contradicts Lemma \ref{lm5} as $N^+_{D[K(C)]}(a_i) \setminus \{{c_i}\}=N^+_{D[K(C)]}(c_i) \setminus \{{a_i}\}$. Let $x_i \in K(C)$. We may assume that $x_iy_i$ is a missing edge in $C$ for some $y_i \in K(C)$.\\
If $x_i \in a_t^- \setminus \{a_{t+1},b_{t+1},c_{t+1},c_t\}$, then $y_i \in a_t^+$ by Corollary \ref{c.p}. Let $x_{i-1}y_{i-1}$ such that $x_{i-1}y_{i-1} \rightarrow x_iy_i$. We may assume that $x_{i-1}\rightarrow x_i$ and $y_{i-1}\rightarrow y_i$, where $y_i \notin x_{i-1}^+\cup x_{i-1}^{++}$. If $x_{i-1} \rightarrow a_t$, then $x_{i-1} \rightarrow a_t \rightarrow y_i$, and hence $y_i \in x_{i-1}^{++}$, which is a contradiction. It follows that $a_t\rightarrow x_{i-1}$. Thus, $a_t\rightarrow x_{i-1}\rightarrow x_i$, and hence $x_i \in a_t^{++}$. Therefore, $a_t^- \setminus \{a_{t+1},b_{t+1},c_{t+1},c_t\} \subseteq a_t^{++}$. Conversely, we will show that $a_t^{++} \setminus \{b_t\}\subseteq a_t^- \setminus \{a_{t+1},b_{t+1},c_{t+1},c_t\}$. If $x \in a_t^{++} \setminus \{b_t\}$, then there is $y \in a_t^+$ such that $ a_t \rightarrow y\rightarrow x$. But $a_tx$ is not a missing edge, so $x \in a_t^-$. We deduce that $a_t^{++} = a_t^- \cup \{b_t\} \setminus \{a_{t+1},b_{t+1},c_{t+1},c_t\}$.\\
(2) Similarly, by symmetry, we prove that $c_t^{++} = c_t^- \cup \{b_t\} \setminus \{a_{t+1},b_{t+1},c_{t+1},a_t\}$.\\
(3) From the losing relations $a_tb_t \rightarrow a_{t+1}b_{t+1}$ and $b_tc_t \rightarrow b_{t+1}c_{t+1}$, we have $\{a_{t+1},b_{t+1},c_{t+1}\} \cap b_t^{++} =\emptyset$ and $a_t,c_t \in b_t^{++}$. Let $x_i \in K(C)$ where $x_iy_i$ is a missing edge in $C$ for some $y_i \in K(C)$.\\
If $x_i \in b_t^- \setminus \{a_{t+1},b_{t+1},c_{t+1}\}$, then $y_i \in b_t^+$ by Corollary \ref{c.p}. Assume that $x_{i-1}y_{i-1} \rightarrow x_iy_i$ such that $x_{i-1}\rightarrow x_i$ and $y_{i-1}\rightarrow y_i$ where $y_i \notin x_{i-1}^+\cup x_{i-1}^{++}$. If $x_{i-1} \rightarrow b_t$ then $x_{i-1} \rightarrow b_t \rightarrow y_i$, and hence $y_i \in x_{i-1}^{++}$, which is a contradiction. It follows that $b_t\rightarrow x_{i-1}$. Thus, $b_t\rightarrow x_{i-1}\rightarrow x_i$, and hence $x_i \in b_t^{++}$. Therefore, $b_t^- \setminus \{a_{t+1},b_{t+1},c_{t+1}\} \subseteq b_t^{++}$. Conversely, if $x \in b_t^{++} \setminus \{a_t,c_t\}$, then there is $y \in b_t^+$ such that $ b_t \rightarrow y\rightarrow x$. But $b_tx$ is not a missing edge, so $x \in b_t^-$. Thus $b_t^{++} \setminus \{a_t,c_t\}\subseteq b_t^- \setminus \{a_{t+1},b_{t+1},c_{t+1}\}$. We deduce that $b_t^{++} = b_t^- \cup \{a_t,c_t\} \setminus \{a_{t+1},b_{t+1},c_{t+1}\}$.
\end{proof}

\begin{lemma}
Let $D$ be a tournament missing disjoint paths of length 2. Let $C=a_1b_1c_1, \ldots, a_kb_kc_k$ be a double cycle in $\Delta(D)$. For all $t \in [1,k]$, in $D[K(C)]$, we have:
\begin{enumerate}
\item
If $c_t \notin a_t^-$, then
$|a_t^{++}|=\left\{\begin{aligned}
 & |a_t^-| & \quad if\  \ b_{t+1} \in a_t^-,\\
 & |a_t^-|-1 & \quad otherwise.
\end{aligned}\right.$\\
\item
If $a_t \notin c_t^-$, then
$|c_t^{++}|=\left\{\begin{aligned}
 & |c_t^-| & \quad if\  \ b_{t+1} \in c_t^-,\\
 & |c_t^-|-1 & \quad otherwise.
\end{aligned}\right.$\\
\item
$|b_t^{++}|=\left\{\begin{aligned}
 & |b_t^-|+1 & \quad if\  \ b_{t+1} \in b_t^-,\\
 & |b_t^-| & \quad otherwise.
\end{aligned}\right.$\\
\end{enumerate}
\end{lemma}
\begin{proof}
(1) By Corollary \ref{c.p}, we have $b_{t+1} \in a_t^-$ if and only if $a_{t+1},c_{t+1} \in a_t^+$, that is, if and only if $a_{t+1},c_{t+1} \notin a_t^-$. Likewise, $b_{t+1} \notin a_t^-$ if and only if $a_{t+1},c_{t+1} \in a_t^-$. Recall that, by Lemma \ref{lm6}, we have $a_t^{++} = a_t^- \cup \{b_t\} \setminus \{a_{t+1},b_{t+1},c_{t+1}\}$.\\
If $b_{t+1} \in a_t^-$, then $|a_t^{++}| = |a_t^-| + |b_t| - |b_{t+1}| = |a_t^-|$.\\
If $b_{t+1} \notin a_t^-$, then $|a_t^{++}| = |a_t^-| + |b_t| - |\{a_{t+1},c_{t+1}\}| = |a_t^-| - 1$.\\
(2) The proof runs as before.\\
(3) By Corollary \ref{c.p}, we have $b_{t+1} \in b_t^-$ if and only if $a_{t+1},c_{t+1} \in b_t^+$, that is, if and only if $a_{t+1},c_{t+1} \notin b_t^-$. Likewise, $b_{t+1} \notin b_t^-$ if and only if $a_{t+1},c_{t+1} \in b_t^-$. Recall that, by Lemma \ref{lm6}, we have $b_t^{++} = b_t^- \cup \{a_t,c_t\} \setminus \{a_{t+1},b_{t+1},c_{t+1}\}$.\\
If $b_{t+1} \in b_t^-$, then $|b_t^{++}| = |b_t^-| + |\{a_{t},c_{t}\}| - |b_{t+1}| = |b_t^-| + 1$.\\
If $b_{t+1} \notin b_t^-$, then $|b_t^{++}| = |b_t^-| + |\{a_{t},c_{t}\}| - |\{a_{t+1},c_{t+1}\}| = |b_t^-|$.
\end{proof}

\indent Now, we are ready to find vertices satisfying the SNP.

\begin{proposition}
\label{a}
Let $D$ be a tournament missing disjoint paths of length 2. Let $C=a_1b_1c_1, \ldots, a_kb_kc_k$ be a double cycle in $\Delta(D)$. There exists $s \in [1,k]$ such that $|N_{D[K(C)]}^+(a_s)| \leq |N_{D[K(C)]}^{++}(a_s)|$, or $|N_{D[K(C)]}^+(c_s)| \leq |N_{D[K(C)]}^{++}(c_s)|$.
\end{proposition}
\begin{proof}
Set $A=\{a_i\colon i\in [1,k]\}$, $A'=\{c_i\colon i\in [1,k]\}$ and $B=\{b_i\colon i\in [1,k]\}$.\\
Let $D[A]$ be the digraph induced by $A$. We have $D[A]$ is a tournament. Hence, by Theorem \ref{t.tourn}, there is a vertex $a_s$ satisfying $SNP$ in $D[A]$. So $|N_{D[A]}^+(a_s)| \leq |N_{D[A]}^{++}(a_s)|$. We may suppose that $c_s \notin a_s^+$ (otherwise we consider $c_s$). Set $|N_{D[A]}^+(a_s)|=m_1$ and $|N_{D[A]}^{++}(a_s)|=m_2$. So $m_1 \leq m_2$.

We compute $|a_s^+|$ and $|a_s^{++}|$ in $K(C)$:\\
By Corollary \ref{c.p}, for all $i \in [1,k]\setminus \{s\}$, we have $a_i \in a_s^+$ if and only if $c_i \in a_s^+$, and consequently $a_i \in a_s^{++}$ if and only if $c_i \in a_s^{++}$. It follows that $|a_s^+ \cap A'| = |a_s^+ \cap A|$ and $|a_s^{++} \cap A'| = |a_s^{++} \cap A|$. Therefore $|(A \cup A')\cap a_s^+|=2m_1$ and $|(A \cup A')\cap a_s^{++}|=2m_2$. Also, by Corollary \ref{c.p}, we have $b_i \in a_s^-$ if and only if $a_i \in a_s^+$. Equivalently, $|B \cap a_s^-|=|A \cap a_s^+|=m_1$
\setcounter{case}{0}
\begin{case}
$a_{s+1} \in a_s^+$.
\end{case}
First we compute the number of $b_i$'s contained in $a_s^{++}$. Since $a_{s+1} \in a_s^+$, we have $b_{s+1} \in a_s^-$ by Corollary \ref{c.p}. Recall that, by Lemma \ref{lm6}, we have $a_s^{++} = a_s^- \cup \{b_s\} \setminus \{a_{s+1},b_{s+1},c_{s+1}\}$. Note that $b_s \notin a_s^-$ since $a_sb_s$ is a missing edge. Thus
\begin{align*}
|B \cap a_s^{++}| &= |B \cap a_s^-|+|\{b_s\}|-|\{b_{s+1}\}|\\
                     &= m_1 + 1 - 1\\
                     &= m_1
\end{align*}
Therefore,
\begin{equation}
\label{a++}
|a_s^{++}|=2m_2 + m_1
\end{equation}

Now we compute the number of $b_i$'s contained in $a_s^+$. We have $a_{s+1} \in a_s^+$ and $b_{s+1} \in a_s^-$; equivalently, we have $a_{s+1} \notin a_s^-$ and $b_{s+1} \notin a_s^+$. Thus $a_i \in a_s^{++}$ if and only if $a_i \in a_s^- \cup \{b_s\} \setminus \{a_{s+1},b_{s+1},c_{s+1}\}$, that is, if and only if $a_i \in a_s^-$ since $a_{s+1} \notin a_s^-$. But $a_i \in a_s^-$ if and only if $b_i \in a_s^+$ by Corollary \ref{c.p}. Note that $b_{s+1} \notin a_s^+$. It follows that $a_i \in a_s^{++}$ if and only if $b_i \in a_s^+$. This means that for each $a_i$ in $a_s^{++}$, we count $b_i$ in $a_s^+$. Thus $|B \cap a_s^+|=|A \cap a_s^{++}|=m_2$. Therefore,
\begin{equation}
\label{a+}
|a_s^+| = 2m_1 + m_2
\end{equation}
Equations (\ref{a++}) and (\ref{a+}) show that $|a_s^+| \leq |a_s^{++}|$.

\begin{case}
$a_{s+1} \notin a_s^+$.
\end{case}
We apply the same reasoning.\\
Since $a_{s+1} \notin a_s^+$, we have $b_{s+1} \notin a_s^-$ by Corollary \ref{c.p}. Thus
\begin{align*}
|B \cap a_s^{++}| &= |B \cap a_s^-|+|\{b_s\}|\\
                     &= m_1 + 1\\
                     &= m_1 + 1
\end{align*}
Therefore,
\begin{equation}
\label{a++2}
|a_s^{++}|=2m_2 + m_1 + 1
\end{equation}

Recall that, equation (\ref{a+}) gives the size of $a_s^+$ in case of $b_{s+1} \notin a_s^+$. Here, we have $a_{s+1} \notin a_s^+$, and hence $b_{s+1} \in a_s^+$ by Corollary \ref{c.p}. So the right-hand side of equation (\ref{a+}) require a simple modification, that is, $a_s^+$ gains only $b_{s+1}$. Thus,
\begin{equation}
\label{a+2}
|a_s^+| = 2m_1 + m_2 + 1
\end{equation}
Again, equations (\ref{a++2}) and (\ref{a+2}) show that $|a_s^+| \leq |a_s^{++}|$.
\end{proof}

\begin{proposition}
\label{b}
Let $D$ be a tournament missing disjoint paths of length 2. Let $C=a_1b_1c_1, \ldots, a_kb_kc_k$ be a double cycle in $\Delta(D)$. There exists $t \in [1,k]$ such that $|N_{D[K(C)]}^+(b_t)| \leq |N_{D[K(C)]}^{++}(b_t)|$.
\end{proposition}
\begin{proof}
Set $A=\{a_i\colon i\in [1,k]\}$, $A'=\{c_i\colon i\in [1,k]\}$ and $B=\{b_i\colon i\in [1,k]\}$.\\
Note that $\displaystyle\sum_{i=1}^{i=k} d_{D[B]}^+(b_i)=\displaystyle\sum_{i=1}^{i=k} d_{D[B]}^-(b_i)= |E({D[B]})|$. It is easy to see that there exists $t \in [1,k]$ such that $|N_{D[B]}^+(b_t)| \geq |N_{D[B]}^-(b_t)|$. In fact, if $|N_{D[B]}^+(b_i)| < |N_{D[B]}^-(b_i)|$ for all $i \in [1,k]$, then $\displaystyle\sum_{i=1}^{i=k} d_{D[B]}^+(b_i) < \displaystyle\sum_{i=1}^{i=k} d_{D[B]}^-(b_i)$, which is a contradiction.\\
Set $|N_{D[B]}^+(b_t)|=n_1$ and $|N_{D[B]}^-(b_t)|=n_2$. So $n_1 \geq n_2$. By Corollary \ref{c.p}, we have $b_i \in b_t^+$ if and only if $a_i,c_i \in b_t^-$ for all $i \in [1,k]$. Hence for each $b_i \in b_t^+$ we count 2 elements, $a_i$ and $c_i$, in $b_t^-$. Equivalently, we have $|(A \cup A')\cap b_t^-|=2n_1$. Thus
\begin{align*}
 |b_t^-| & = |(A \cup A')\cap b_t^-|+|B \cap b_t^-|\\
         & = 2n_1 + n_2
\end{align*}

Compute $|b_t^{+}|$:\\
By Corollary \ref{c.p}, we have $b_i \in b_t^-$ if and only if $a_i,c_i \in b_t^+$ for all $i \in [1,k]$. Hence for each $b_i \in b_t^-$ we count 2 elements, $a_i$ and $c_i$, in $b_t^+$. Equivalently, we have $|(A \cup A')\cap b_t^+|=2n_2$. Thus
\begin{align*}
 |b_t^+| & = |(A \cup A')\cap b_t^+| + |B \cap b_t^+|\\
         & = 2n_2 + n_1
\end{align*}

Compute $|b_t^{++}|$. By Lemma \ref{lm6}, we have $b_t^{++} = b_t^- \cup \{a_t,c_t\} \setminus \{a_{t+1},b_{t+1},c_{t+1}\}$. There are two cases:
\setcounter{case}{0}
\begin{case}
If $b_{t+1} \in b_t^+$, then $a_{t+1}, c_{t+1} \in b_t^-$ by Corollary \ref{c.p}. Thus,
\begin{align*}
   |b_t^{++}| & = |b_t^-| + |\{{a_t,c_t}\}| - |\{{a_{t+1}, c_{t+1}}\}|\\
              & = 2n_1 + n_2 + 2 - 2\\
              & = 2n_1 + n_2
\end{align*}
\end{case}
\begin{case}
If $b_{t+1} \in b_t^-$, then $a_{t+1}, c_{t+1} \notin b_t^-$ by Corollary \ref{c.p}. Thus,
\begin{align*}
   |b_t^{++}| & = |b_t^-| + |\{{a_t,c_t}\}| - |\{{b_{t+1}}\}|\\
              & = 2n_1 + n_2 + 2 - 1\\
              & = 2n_1 + n_2 + 1
\end{align*}
\end{case}

In both cases, we have $2n_2 + n_1 \leq 2n_1 +n_2$ and $2n_2 + n_1 \leq 2n_1 +n_2 + 1$ since $n_1 \geq n_2$. Therefore $|b_t^{+}| \leq |b_t^{++}|$.
\end{proof}

\subsubsection{Case of $\Delta$ is 2-regular}
\begin{lemma}\cite{fidler2007remarks}
\label{k.0}
Let $D$ be a tournament missing a matching. Let $C=a_1b_1,\ldots,a_kb_k$ be a directed cycle of $\Delta(D)$ such that $a_i \rightarrow a_{i+1}$ and $b_i \rightarrow b_{i+1}$ for all $i \in [1,k-1]$.\\
i) If $k$ is odd, then $a_k \rightarrow a_1$ and $b_k \rightarrow b_1$.\\
ii) If $k$ is even, then $a_k \rightarrow b_1$ and $b_k \rightarrow a_1$.
\end{lemma}

Note that if $C$ is a directed cycle in $\Delta(D)$, then $C$ is also a directed cycle in $\Delta(D[K(C)])$. So we can modify Lemma \ref{k.0} as follows.

\begin{lemma}
\label{k.1}
Let $D$ be an oriented graph. Let $C=a_1b_1,\ldots,a_kb_k$ be a directed cycle of $\Delta(D)$ such that $a_i \rightarrow a_{i+1}$ and $b_i \rightarrow b_{i+1}$ for all $i \in [1,k-1]$. Suppose that $D[K(C)]$ is a tournament missing a matching.\\
i) If $k$ is odd, then $a_k \rightarrow a_1$ and $b_k \rightarrow b_1$.\\
ii) If $k$ is even, then $a_k \rightarrow b_1$ and $b_k \rightarrow a_1$.
\end{lemma}

\begin{proposition} \label{DC-int}
Let $D$ be a tournament missing disjoints paths of length 2 and let $\Delta$ denotes its dependency digraph. If $C$ is a double cycle in $\Delta$, then $K(C)$ is an interval of $D$. That is, for all $ u,v \in K(C)$, we have $N^+(u)\setminus K(C) = N^+(v)\setminus K(C)$ and $N^-(u)\setminus K(C) = N^-(v)\setminus K(C)$.
\end{proposition}
\begin{proof}
Let $C=a_1b_1c_1, \ldots, a_kb_kc_k$ be a double cycle in $\Delta$. For all $i \in [1,k-1]$, set $\{{x_i,y_i}\}=\{{a_i,b_i}\}$ so that $x_i \rightarrow x_{i+1}$ and $y_i \rightarrow y_{i+1}$. We have $C_1=x_1y_1,\ldots,x_ky_k$ is a cycle in $\Delta$. Note that for all $w \notin K(C)$, we have $w$ is adjacent to every vertex in $K(C)$.

If $x_1 \rightarrow w$ for some $w \notin K(C)$, then $y_2 \rightarrow w$ since otherwise $x_1 \rightarrow w \rightarrow y_2$, which contradicts the fact that $x_1y_1 \rightarrow x_2y_2$. So $N^+(x_1) \setminus K(C) \subseteq N^+(y_2) \setminus K(C)$. By applying this argument to every losing relation in $C$, we get $N^+(x_i) \setminus K(C) \subseteq N^+(y_{i+1}) \setminus K(C)$ for all $ 1 \leq i \leq k-1$. Similarly, for all $ 1 \leq i \leq k-1$, we have $N^+(y_i) \setminus K(C) \subseteq N^+(x_{i+1}) \setminus K(C)$. If $k$ is even, then $x_k \rightarrow y_1$ and $y_k \rightarrow x_1$ by Lemma \ref{k.1}. Hence we obtain that $N^+(y_k) \setminus K(C) \subseteq N^+(y_1) \setminus K(C)$ and $N^+(x_k) \setminus K(C) \subseteq N^+(x_1) \setminus K(C)$. It follows that $N^+(x_1) \setminus K(C) \subseteq N^+(y_2) \setminus K(C) \subseteq\cdots \subseteq N^+(y_k) \setminus K(C) \subseteq N^+(y_1) \setminus K(C) \subseteq N^+(x_2) \setminus K(C)\subseteq\cdots\subseteq N^+(x_k) \setminus K(C) \subseteq N^+(x_1) \setminus K(C)$. Therefore all inclusions are equalities. If $k$ is odd, then $x_k \rightarrow x_1$ and $y_k \rightarrow y_1$ by Lemma \ref{k.1}. Hence $N^+(x_k) \setminus K(C) \subseteq N^+(y_1) \setminus K(C)$ and $N^+(y_k) \setminus K(C) \subseteq N^+(x_1) \setminus K(C)$. It follows that $N^+(x_1) \setminus K(C) \subseteq N^+(y_2) \setminus K(C) \subseteq\cdots\subseteq N^+(x_k) \setminus K(C) \subseteq N^+(y_1) \setminus K(C) \subseteq N^+(x_2) \setminus K(C)\subseteq\cdots\subseteq N^+(y_k) \setminus K(C) \subseteq N^+(x_1) \setminus K(C)$. Thus all inclusions are equalities. Therefore for all $ u, v \in K(C_1)$, we have $N^+(u)\setminus K(C)=N^+(v)\setminus K(C)$. In the same manner we can see that $N^-(u)\setminus K(C)=N^-(v)\setminus K(C)$ for all $ \ u,v \in K(C_1)$.

Likewise, we set $\{{x'_i,y'_i}\}=\{{b_i,c_i}\}$, where $C_2=x'_1y'_1,\ldots,x'_ky'_k$ is a cycle in $\Delta$ so that $x'_i \rightarrow x'_{i+1}$ and $y'_i \rightarrow y'_{i+1}$. Similar considerations apply to $C_2$. Thus for all $ \ u, v \in K(C_2)$, we have $N^+(u)\setminus K(C)=N^+(v)\setminus K(C)$ and $N^-(u)\setminus K(C)=N^-(v)\setminus K(C)$. It follows that $N^+(u)\setminus K(C)=N^+(v)\setminus K(C)$ and $N^-(u)\setminus K(C)=N^-(v)\setminus K(C)$ for all $ \ u, v \in K(C)$. Therefore $K(C)$ is an interval of $D$.
\end{proof}

\begin{theorem}
Let $D$ be a tournament missing disjoints paths of length 2 and let $\Delta$ denote its dependency digraph.
If $d_{\Delta}^+(ab)=d_{\Delta}^-(ab)=2$ for all $ab \in V(\Delta)$, then $D$ has a vertex satisfying the SNP.
\end{theorem}
\begin{proof}
Since $d_{\Delta}^+(ab)=d_{\Delta}^-(ab)=2$, the dependency digraph $\Delta$ of $D$ is composed of double cycles only.
For every $v \in V(D)$ we have $J(v)=\{{v}\}$ or $J(v)=K(C)$ for some double cycle $C$ in $\Delta$. For every double cycle $C$ in $\Delta$, we have $K(C)$ is an interval of $D$ by Proposition \ref{DC-int}. Hence $D$ is a good digraph. So we can apply Theorem \ref{th1}. Let $L$ be a good median order of $D$ and let $f$ denote its feed vertex.\\
\indent If $J(f)=K(C)$ for some double cycle $C$ in $\Delta$, then there exists a vertex $v$ satisfying the SNP in $D[J(f)]$ by Propositions \ref{a} and \ref{b}. Therefore, by Theorem \ref{th1}, $v$ has the SNP in $D$.\\
\indent If $J(f)=\{f\}$, then clearly $f$ has the SNP in $D[J(f)]$. Thus by Theorem \ref{th1}, $f$ has the SNP in $D$.
\end{proof}
\subsection{SSNC in Tournaments missing disjoint paths of length at most 2}
\label{ssnc_leq2}

\begin{lemma}\cite{fidler2007remarks}
\label{l.cycle.interval0}
Let $D$ be a tournament missing a matching. If $C=a_1b_1,\ldots,a_kb_k$ is a directed cycle of $\Delta(D)$, then $K(C)$ is an interval of $D$.
\end{lemma}
 We need to make a slight modification in the statement of Lemma \ref{l.cycle.interval0}.
\begin{lemma}
\label{l.cycle.interval}
Let $D$ be a digraph. Let $\{a_1b_1,\ldots,a_kb_k\}$ be a set of disjoint missing edges in $D$ such that $C=a_1b_1,\ldots,a_kb_k$ is a directed cycle of $\Delta(D)$. If $K(C)=J(x)$ for some $x \in V(D)$, then $K(C)$ is an interval of $D$.
\end{lemma}

\begin{proof}
Since $K(C)=J(x)$, every vertex in $D \setminus K(C)$ is adjacent to each vertex in $K(C)$. We orient all the missing edges in $D \setminus K(C)$. Hence we obtain a new digraph $D'$ such that $V(D')=V(D)$. It is clear that $D'$ is a tournament missing a matching. Furthermore, we have $C$ is a directed cycle in $\Delta(D')$, since the losing relations of $C$ do not modify. Hence, by Lemma \ref{l.cycle.interval0}, $K(C)$ is an interval of $D'$; this means that for all $ u,v \in K(C)$, we have $N_{D'}^+(u)\setminus K(C) = N_{D'}^+(v)\setminus K(C)$ and $N_{D'}^-(u)\setminus K(C) = N_{D'}^-(v)\setminus K(C)$. But, for all $u \in K(C)$, we have $N_{D'}^+(u)=N_D^+(u)$  and $N_{D'}^-(u)=N_D^-(u)$ since $u$ is not incident to any new arc. Hence for all $ u,v \in K(C)$, we have $N_{D}^+(u)\setminus K(C) = N_{D}^+(v)\setminus K(C)$ and $N_{D}^-(u)\setminus K(C) = N_{D}^-(v)\setminus K(C)$. Therefore $K(C)$ is an interval of $D$.
\end{proof}

\begin{lemma}\cite{ghazal2015remark}
\label{l.cycle.snp}
Let $\{a_1b_1,\ldots,a_kb_k\}$ be a set of disjoint missing edges in $D$. If $C=a_1b_1,\ldots,a_kb_k$ is a directed cycle of $\Delta$, then every vertex in $K(C)$ satisfy the SNP in $D[K(C)]$.
\end{lemma}

\begin{theorem}
Let $D$ be a digraph missing disjoint paths of length at most 2. If the missing disjoint paths of length 2 form double cycles in $\Delta(D)$, then $D$ has a vertex with the SNP.
\end{theorem}
\begin{proof}
For every $J(x)$ of $D$, we will show that $J(x)$ is an interval of $D$ containing a vertex with the SNP in $D[J(x)]$, or $J(x)=K(P)$ such that $P$ is a maximal directed path in $\Delta(D)$. Then we apply Theorem \ref{th_good_completion} to conclude that $D$ has a vertex with the SNP. For a whole vertex $x$, we have $J(x) = \{x\}$ and there nothing to prove. Recall that, by Proposition \ref{p.in-out_leq_2}, every missing edge of $D$ has in- and out-degree at most 2 in $\Delta(D)$. Moreover by Lemma~\ref{lm1}, if a missing edge $uv$ having out-degree 2 in $\Delta(D)$, then their out-neighbors are two missing edges which share a common vertex in $D$; in other words, the two out-neighbors of $uv$ form a missing path of length 2. Similarly, by Lemma \ref{lm2}, if $d_{\Delta}^-(uv) = 2$, then their two in-neighbors form a missing path of length 2. For every missing edge $ab$ containing in a double cycle $C$, we have $d_{\Delta}^-(ab)=d_{\Delta}^+(ab)=2$. Hence, the missing edge $ab$ has no in-neighbors and no out-neighbors outside of $C$. This means that $C$ is a connected component in $\Delta$. Furthermore, for all missing edges $uv$ not containing in $C$, we have $K(C) \cap \{u,v\} = \emptyset$ since $D$ is a digraph missing disjoint paths of length at most 2. Hence $K(C)=J(x)$ for some $x \in V(D)$. It is clear that $D[K(C)]$ is a digraph missing disjoint paths of length 2 where $C$ is also a double cycle in $\Delta(D[K(C)])$. So, by Proposition \ref{a}, there is a vertex with the SNP in $D[K(C)]$. By Proposition \ref{DC-int}, we can easily deduce that $K(C)$ is an interval of $D$. Actually, Proposition \ref{DC-int} asserts that $K(C)$ is an interval in case of digraph missing disjoint paths of length exactly 2. Here, since $K(C) = J(x)$ for some $x \in V(D)$, we can safely orient (arbitrary) the missing paths of length 1 without modifying the in- and out-neighborhoods as well as the losing relations within $K(C)$. Hence we obtain a new digraph $D'$ with $V(D') = V(D)$, which is a digraph missing disjoint paths of length 2, and $C$ is also a double cycle in $\Delta(D')$. Now, by Proposition \ref{DC-int}, $K(C)$ is an interval of $D'$. As $N_D^+(u) = N_{D'}^+(u)$ and $N_D^-(u) = N_{D'}^-(u)$, we get $K(C)$ is also an interval of $D$. Now, we return to the initial digraph $D$. Because any missing edge containing in a missing path of length 2 has no in-neighbors and no out-neighbors outside of its double cycle, we deduce that the remaining missing edges (which are missing paths of length 1) have in- and out-degrees at most 1 in $\Delta$. This means that these missing paths of length 1 form disjoint directed paths and directed cycles in $\Delta$. It is clear that if $Q$ is a directed path or a directed cycle in $\Delta$, then $K(Q) = J(x)$ for some $x \in V(D)$. For every directed cycle $C$ in $\Delta$, we have $K(C)$ is an interval of $D$ and has a vertex with the SNP in $D[K(C)]$ by Lemmas \ref{l.cycle.interval} and \ref{l.cycle.snp}. Now, we can apply Theorem \ref{th_good_completion} and deduce that $D$ has a vertex with the SNP.
\end{proof}

A natural question is to seek more than one vertex with the SNP. Havet and Thomass\'e used the sedimentation to exhibit a second vertex with the SNP in tournaments that do not have any sink. Recall that if $L$ is a good weighted median order of a good digraph $(D,\omega)$, then the sedimentation of $L$ is also a good weighted median order of $(D,\omega)$. Define now inductively $Sed^0(L) = L$ and $Sed^{q+1}(L) = Sed(Sed^q(L))$. If the process reaches a rank $q$ such that $Sed^q(L) = y_1\ldots y_n$ and $\omega(N^+(y_n) \setminus J(y_n)) < \omega(G_{{Sed}^q}(L) \setminus J(y_n))$, call the order $L$ \emph{stable}. Otherwise call $L$ \emph{periodic}. We will use these new orders to exhibit at least two vertices with the SNP in some cases.

\begin{theorem}\cite{ghazal2015remark}
Let $D$ be an oriented graph missing a matching and suppose that its dependency digraph $\Delta$ is composed of only directed cycles. If $D$ has no sink vertex, then it has at least two vertices with the SNP.
\end{theorem}

Using same arguments of the proof of the previous result, we can generalize it as follows:

\begin{theorem}
\label{t.nosink}
Let $D$ be a good digraph. Suppose that for every vertex $x$ incident to a missing edge, we have $J(x)$ contains at least two vertices with the SNP in $D[K(J(x))]$. If $D$ has no sink vertex, then it has at least two vertices with the SNP.
\end{theorem}
\begin{proof}
Consider a good median order $L = x_1 \ldots x_n$ of $D$. If $x_n$ is incident to a missing edge, then $J(x_n)$ contains at least two vertices with the SNP in $D[K(J(x_n))]$. Hence, by Theorem \ref{th1}, the result holds. Otherwise, $x_n$ is a whole vertex. So $J(x_n) = {x_n}$. We have $x_n$ has the SNP in $D$ by Theorem \ref{th1}. So we need to find another vertex with the SNP. Consider the good median order $L' = x_1\ldots x_{n-1}$ of $D \setminus \{x_n\}$. If $L'$ is stable, then there is $q$ for which $Sed^q(L') = y_1\ldots y_{n-1}$ and $|N^+(y_{n-1}) \setminus J(y_{n-1})| < | G_{Sed^q}(L') \setminus J(y_{n-1})|$. Note that $y_1\ldots y_{n-1}x_n$ is also a good median order of $D$. There exists $y \in J(y_{n-1})$ such that $y$ has the SNP in $D[J(y_{n-1})]$. So $| N^+(y) |=| N^+_{D[y_1,y_{n-1}]}(y) \setminus J(y) | + |N^+(y) \cap J(y)| +1 \leq |G_{Sed^q}(L') \setminus J(y)| + |N^+(y) \cap J(y)| \leq |N^{++}(y) \setminus J(y)| + |N^{++}(y) \cap J(y)| = |N^{++}(y)|$. Now suppose that $L'$ is periodic. Since $D$ has no sink, the vertex $x_n$ has an out-neighbor $x_j$. Choose $j$ to be the greatest (so that it is the last vertex of its corresponding interval). Note that for every $q$, we have $x_n$ is an out-neighbor of the feed vertex of $Sed^q(L')$. So $x_j$ is not the feed vertex of any $Sed^q(L')$. Since $L'$ is periodic, the vertex $x_j$ must be a bad vertex of $Sed^q(L')$ for some integer $q$, otherwise the index of $x_j$ would always increase during the sedimentation process. Let $q$ be such an integer. Set $Sed^q(L') = y_1\ldots y_{n-1}$. There exists $y \in J(y_{n-1})$ such that $y$ has the SNP in $D[J(y_{n-1})]$. Note that $y \rightarrow x_n \rightarrow x_j$ and $(G_{Sed^q}(L') \setminus J(y)) \cup \{x_j\} \subseteq N^{++}(y) \setminus J(y)$. So $|N^+(y)| = |N^+_{D[y_1,y_{n-1}]}(y) \setminus J(y)| + 1 + |N^+(y) \cap J(y)| = |G_{Sed^q}(L') \setminus J(y)| + 1 + |N^+(y) \cap J(y)| =  |G_{Sed^q}(L') \setminus J(y)| + |\{x_j\}| + |N^+(y) \cap J(y)| \leq |N^{++}(y) \setminus J(y)| + |N^{++}(y) \cap J(y)| = |N^{++}(y)|$.
\end{proof}

\begin{theorem}
Let $D$ be a digraph missing disjoint paths of length at most 2. Suppose that $\Delta(D)$ is composed of only directed cycles and double cycles. If $D$ has no sink vertex, then it has at least two vertices with the SNP.
\end{theorem}
\begin{proof}
For every $x$ which is incident to a missing edge, we have $J(x)$ is a cycle or double cycle, and hence $J(x)$ contains at least two vertices with the SNP in $D[K(J(x))]$ by Lemma \ref{l.cycle.snp} and Propositions \ref{a} and \ref{b}. Therefore, by Theorem \ref{t.nosink}, there are at least two vertices with the SNP.
\end{proof}

\section{Proof of Theorem \ref{th_good_completion}}
\label{proof}
\begin{proof}[Proof of Theorem \ref{th_good_completion}]
For $x\in V(D)$, if $x$ is a whole vertex then $J(x)=\{x\}$. Otherwise, we have either $J(x)$ is an interval of $D$ or $J(x)=K(P)$ for some directed path $P$ (connected component) in $\Delta$. Recall that for all $u\in V(D) \setminus J(x)$, we have $u$ is adjacent to every vertex that appear in $J(x)$. Hence, orienting missing edges outside of $J(x)$ has no influence on $J(x)$ regarding the in- and out-neighborhoods of the vertices as well as the losing relations within $J(x)$. Thus by orienting all the missing edges that appear outside of the collection of the $J(x)$'s that are intervals of $D$, we obtain a new digraph $D_1$ such that for every $J(x)$ in $D_1$, we have $J(x)=\{x\}$ which is a trivial interval of $D_1$, or $J(x)$ is an interval of both $D$ and $D_1$. This means that $D_1$ is a good digraph. We will use this fact to create a particular good completion of $D$. In fact, starting from $D$, we take a $J(x)$ such that $J(x)$ is not an interval of $D$; that is $J(x)=K(P)$ for some directed path $P=a_1b_1,\ldots,a_kb_k$ in $\Delta(D)$, namely $a_i \rightarrow a_{i+1}$ and $b_i \rightarrow b_{i+1}$ for $i=1,\ldots,k-1$. Note that, as a connected component, $P$ must be a maximal directed path in $\Delta$. So $a_1b_1$ is a good missing edge since $d_{\Delta}^-(a_1b_1)=0$. We may assume without loss of generality that $(a_1, b_1)$ is a convenient orientation of $a_1b_1$. We orient $a_ib_i$ as $(a_i,b_i)$ for $i=1,\ldots,k-1$. We follow the same method of orientation for the $J(x)$'s that are not intervals of $D$. We denote by $F$ the set of the new arcs added to $D$. Set $D'=D+F$. Hence $D'$ is a good completion of $D$. Let $L$ be a good median order of the good digraph $D'$ and let $f$ denote its feed vertex. By Theorem \ref{th1}, for all $x \in J(f)$ we have $|N_{D'}^+(x) \setminus J(f)| \leq |G_L^{D'} \setminus J(f)| \leq |N_{D'}^{++}(x) \setminus J(f)|$.
\setcounter{case}{0}
\begin{case}
$f$ is not incident to any new arc of $F$.\\
In this case $f$ is a whole vertex, or $J(f)$ is an interval both in $D$ and $D'$. If $J(f)=\{f\}$, then $f$ has the SNP in $D'[J(f)]=D[J(f)]$. If $J(f)$ is an interval, then there exists $ p \in J(f)$ such that $p$ has the SNP in $D'[J(f)]=D[J(f)]$. So
$|N_{D'}^+(p) \cap J(f)| \leq |N_{D'}^{++}(p) \cap J(f)|$.
Since $f$ is not incident to any new arc, every element in $J(f)$ is not incident to any new arc. So $N_{D'}^+(p)=N_D^+(p)$.
We need to show that $N_{D'}^{++}(p) \setminus J(f) \subseteq N_{D}^{++}(p) \setminus J(f)$. Let $v \in N_{D'}^{++}(p) \setminus J(f)$. There exists $x \in V(D)$ such that $p \rightarrow x \rightarrow v \rightarrow p$ in $D'$, where $p \rightarrow x$ and $v \rightarrow p$ in $D$ since $p$ is not incident to any new arc. If $xv$ is not a missing edge of $D$, then $v \in N_{D}^{++}(p)$. Assume now that $xv$ is a missing edge of $D$. If $xv$ is good, then $x \rightarrow v$ is a convenient orientation. Since $p \rightarrow x$, we get $p \rightarrow v$ in $D$ or $v \in N_D^{++}(p)$, by the definition of the convenient orientation. But $v \rightarrow p$, hence we must have $v \in N_D^{++}(p)$. If $xv$ is not good, then there is a missing edge $rs$ such that $rs \rightarrow xv$, namely $s \rightarrow v$ and $x \notin N^+(s) \cup N^{++}(s)$. Note that $ps$ is not a missing edge. As $p \rightarrow x$, we must have $p \rightarrow s$ since otherwise $s \rightarrow p \rightarrow x$, which is a contradiction to the fact that $rs \rightarrow xv$. Hence we get $p \rightarrow s \rightarrow v$. Thus $v \in N_D^{++}(p)$. So $|N_{D'}^{++}(p) \setminus J(f)| \leq |N_{D}^{++}(p) \setminus J(f)|$. Now, we can compare $|N^+(p)|$ and $|N^{++}(p)|$. In fact, $|N_D^+(p)|=|N_{D'}^+(p)|=|N_{D'}^+(p) \setminus J(f)|+|N_{D'}^+(p) \cap J(f)| \leq |N_{D'}^{++}(p) \setminus J(f)| + |N_{D'}^{++}(p) \cap J(f)| \leq |N_{D}^{++}(p) \setminus J(f)| + |N_{D}^{++}(p) \cap J(f)|=|N_{D}^{++}(p)|$. We conclude that if $f$ is not incident to any new arc of $F$, then there exists $p \in J(f)$ such that $p$ has the SNP in $D'$ and in $D$.
\end{case}
\begin{case}
$f$ is incident to a new arc of $F$.\\
In this case $J(f)= K(P)$, for some path $P$ in $\Delta$, which is also a connected component of $\Delta$. Set $P=a_1b_1,\ldots,a_kb_k$, namely $a_i \rightarrow a_{i+1}$, $b_i \rightarrow b_{i+1}$ for $i=1,\ldots,k-1$. So $f=a_t$ or $f=b_t$. We may suppose, without loss of generality, that $a_i \rightarrow b_i$ in $D'$ for all $ \ i=1,\ldots,k$.
\begin{subcase}
Suppose that $f=a_t$ with $t<k$.\\
Since $a_t \rightarrow b_t$ in $D'$, we have $N_{D'}^+(f)=N_D^+(f) \cup \{b_t\}$. So $|N_D^+(f)|=|N_{D'}^+(f)|-1$. And since $a_{t+1} \rightarrow b_{t+1}$ in $D'$, we have $a_t \rightarrow a_{t+1} \rightarrow b_{t+1} \rightarrow a_t$. Therefore $b_{t+1} \in N_{D'}^{++}(f)$.\\

\noindent{\bf{Claim}.} \emph{We have $N_{D'}^{++}(f) \setminus \{b_{t+1}\} \subseteq N_{D}^{++}(f)$.}
\begin{proof}
Let $v \in N_{D'}^{++}(f) \setminus \{b_{t+1}\}$. There is a vertex $x$ such that $f \rightarrow x \rightarrow v \rightarrow f$ in $D'$.

Suppose first that $x \neq b_t$. Note that $(a_t,b_t)$ is the unique new arc of $F$ that is incident to $f=a_t$. Hence $f \rightarrow x$ and $v \rightarrow f$ are in $D$. If $xv$ is not a missing edge of $D$, then $x \rightarrow v$ in $D$, and hence $v \in N_{D}^{++}(f)$. Assume now that $xv$ is a missing edge of $D$. If $xv$ is good, then $x \rightarrow v$ is a convenient orientation. Since $f \rightarrow x$, we have $f \rightarrow v$ or $v \in N_D^{++}(f)$ by the definition of the convenient orientation. But $v \rightarrow f$, so we must have $v \in N_D^{++}(f)$. If $xv$ is not good, then there is a missing edge $rs$ such that $rs \rightarrow xv$, namely $s \rightarrow v$ and $x \notin N^+(s) \cup N^{++}(s)$. Note that $fs$ is not a missing edge, but $f \rightarrow x$, so we must have $f \rightarrow s$ since otherwise $s \rightarrow f \rightarrow x$, which contradicts $rs \rightarrow xv$. Thus we get $f \rightarrow s \rightarrow v$. Therefore $v \in N_D^{++}(f)$.

Suppose now that $x=b_t$. We have $v \neq b_{t+1}$, hence $va_{t+1}$ is not a missing edge of $D$. Furthermore, we must have $a_{t+1} \rightarrow v$ since otherwise ${x}={b_t} \rightarrow v \rightarrow a_{t+1}$ in $D$, which is a contradiction to the fact that $a_tb_t \rightarrow a_{t+1}b_{t+1}$. Thus ${f}={a_t} \rightarrow a_{t+1} \rightarrow v$, and hence $v \in N_{D}^{++}(f)$. Therefore $N_{D'}^{++}(f) \setminus \{b_{t+1}\} \subseteq N_{D}^{++}(f)$.
\end{proof}

\noindent Thus $|N_{D'}^{++}(f)|-1 \leq |N_{D}^{++}(f)|$. Since $J(f)=\{f\}$ in $D'$, we get $|N_{D'}^+(f)| \leq |N_{D'}^{++}(f)|$ by Theorem \ref{th1} . Therefore $|N_D^+(f)|=|N_{D'}^+(f)|-1 \leq |N_{D'}^{++}(f)|-1 \leq |N_{D}^{++}(f)|$.
\end{subcase}
\begin{subcase}
Suppose that $f=a_k$.\\
We reorient the missing edge $a_kb_k$ as $b_k \rightarrow a_k$. Let $D{''}$ denote the new oriented graph. Note that $L$ is a good median order of the good oriented graph $D{''}$, since $a_k \rightarrow b_k$ is a backward arc (directed from right to left) in $D'$. Clearly, $N_{D{''}}^+(f)=N^+(f)$. So $|N^+(f)|=|N_{D{''}}^+(f)|$ and $J(f)=\{f\}$ in $D{''}$. Moreover, $f$ has the SNP in $D{''}$. Thus $|N_{D{''}}^+(f)| \leq |N_{D{''}}^{++}(f)|$. We need to show that $N_{D{''}}^{++}(f) \subseteq N_D^{++}(f)$. Let $v \in N_{D{''}}^{++}(f)$. There exists $x \in V(D)$ such that $f \rightarrow x \rightarrow v \rightarrow f$ in $D{''}$, where $f \rightarrow x$ and $v \rightarrow f$ in $D$. If $xv$ is not a missing edge of $D$, then $x \rightarrow v$ in $D$, hence $v \in N_{D}^{++}(f)$. Assume now that $xv$ is a missing edge of $D$. If $xv$ is good, then $x \rightarrow v$ is a convenient orientation. Since $f \rightarrow x$, we have $f \rightarrow v$ or $v \in N_D^{++}(f)$ by the definition of the convenient orientation. But $v \rightarrow f$, so we must have $v \in N_D^{++}(f)$. If $xv$ is not good, then there is a missing edge $rs$ such that $rs \rightarrow xv$, namely $s \rightarrow v$ and $x \notin N^+(s) \cup N^{++}(s)$. Note that $fs$ is not a missing edge. Since $f \rightarrow x$, we must have $f \rightarrow s$ because otherwise $s \rightarrow f \rightarrow x$, which is a contradiction to the fact that $rs \rightarrow xv$. Hence $f \rightarrow s \rightarrow v$. Thus $v \in N_D^{++}(f)$. So $|N_{D{''}}^{++}(f)| \leq |N_D^{++}(f)|$. Therefore, $|N_D^+(f)|=|N_{D{''}}^+(f)| \leq |N_{D{''}}^{++}(f)| \leq |N_{D}^{++}(f)|$.
\end{subcase}
\begin{subcase}
Suppose that $f=b_t$.\\
Since $N_{D{'}}^+(f)=N^+(f)$, we have $|N^+(f)|=|N_{D{'}}^+(f)|$ and $J(f)=\{f\}$ in $D{'}$. Moreover, $f$ has the SNP in $D{'}$. Hence $|N_{D{'}}^+(f)| \leq |N_{D{'}}^{++}(f)|$. We need to show that $N_{D{'}}^{++}(f) \subseteq N_D^{++}(f)$. Let $v \in N_{D{'}}^{++}(f)$. So there exists $x \in V(D)$ such that $f \rightarrow x \rightarrow v \rightarrow f$ in $D{'}$. Note that $f \rightarrow x$ and $v \rightarrow f$ in $D$.\\
\indent 1) If $xv$ is not a missing edge of $D$, then $x \rightarrow v$ in $D$, and hence $v \in N_{D}^{++}(f)$.\\
\indent 2) If $xv$ is a missing edge of $D$, then:\\
\indent \indent   i) If $xv$ is good, then $x \rightarrow v$ is a convenient orientation. As $f \rightarrow x$, we get $f \rightarrow v$ or $v \in N_D^{++}(f)$ by the definition of the convenient orientation. But $v \rightarrow f$, so we must have $v \in N_D^{++}(f)$.\\
\indent \indent   ii) If $xv$ is not good, then there is a missing edge $rs$ such that $rs \rightarrow xv$, namely $s \rightarrow v$ and $x \notin N^+(s) \cup N^{++}(s)$. Note that $fs$ is not a missing edge. As $f \rightarrow x$, we must have $f \rightarrow s$ since otherwise $s \rightarrow f \rightarrow x$, which is a contradiction to the fact that $rs \rightarrow xv$. So $f \rightarrow s \rightarrow v$. Thus $v \in N_D^{++}(f)$. It follows that $|N_{D{'}}^{++}(f)| \leq |N_D^{++}(f)|$. Therefore $|N_D^+(f)|=|N_{D{'}}^+(f)|  \leq |N_{D{'}}^{++}(f)|  \leq |N_{D}^{++}(f)|$.
\end{subcase}
\end{case}
Finally, $f$ has the SNP in $D'$ and  in $D$.
\end{proof}
\noindent Note that, $p=f$ or $p$ is any vertex in $J(f)$ satisfying the SNP in $D[J(f)]$ when $J(f)$ is an interval of $D$.
\subsection*{Acknowledgements}
The authors thank Pr. Amine El Sahili for his useful comments and remarks. This work is done during the PhD thesis preparation of Moussa Daamouch under the supervision of Pr. Amine El Sahili and Dr. Salman Ghazal.

\end{document}